\documentclass[a4paper]{amsart}

\usepackage[utf8]{inputenc}
\usepackage[english]{babel}

\usepackage[T1]{fontenc}
\usepackage{lmodern}  

\usepackage{my-math-symbol}
\usepackage{my-cleveref}
\usepackage{my-equation-numbering}
\usepackage{tikz}
\usetikzlibrary{cd}

\theoremstyle{plain}
\newtheorem{keyassump}{Assumption}
\newenvironment{keyassumption}[1]{\renewcommand \thekeyassump{#1} \keyassump}{\endkeyassump}
\crefname{keyassumption}{Assumption}{Assumption}

\DeclareMathOperator{\CSC}{CSC}
\DeclareMathOperator{\GModuli}{\calM}
\newcommand{\CC}{\mathrm{CC}}
\newcommand{\Hermform}{\begin{pmatrix} - 1 & 0 \\ 0 & I_{n + 1} \end{pmatrix}}

\usepackage[non-sorted-cites,initials,alphabetic,nobysame]{amsrefs}

\AtBeginDocument{%
	\def\MR#1{}
}

\title[Non-homothetic complete periodic contact forms]{Non-homothetic complete periodic contact forms with constant Tanaka--Webster scalar curvature}

\author{Jeffrey S. Case}
\address{Department of Mathematics \\ The Pennsylvania State University
	\\ University Park, PA 16802, USA}
\email{jscase@psu.edu}

\author{Yuya Takeuchi}
\address{Division of Mathematics \\ Institute of Pure and Applied Sciences \\ University of Tsukuba
	\\ Tsukuba, Ibaraki 305-8571, Japan}
\email{ytakeuchi@math.tsukuba.ac.jp, yuya.takeuchi.math@gmail.com}



\thanks{J.S.C. was supported by the National Science Foundation under Award \#DMS-2505606. Y.T. was supported by JSPS KAKENHI Grant Number JP25K17247.}

\begin{document}

\begin{abstract}
	We study the existence problem for complete contact forms with constant Tanaka--Webster scalar curvature
	on non-compact strictly pseudoconvex CR manifolds.
	We prove that,
	under mild assumptions,
	the universal cover of a compact strictly pseudoconvex CR manifold admits
	infinitely many non-homothetic such contact forms
	whenever its fundamental group has infinite profinite completion.
	As applications,
	we treat complements of real or complex spheres in the standard CR sphere,
	as well as circle bundles over compact K\"{a}hler manifolds and the boundary of a Reinhardt domain.
\end{abstract}

\maketitle

\section{Introduction}
\label{sec:Introduction}

The existence of complete metrics with constant scalar curvature
is one of the most fundamental problems in Riemannian and conformal geometry.
In the compact case,
this problem, known as the Yamabe problem, was solved affirmatively;
see \cite{Lee-Parker1987} and the references therein.
A typical non-compact setting arises
when one considers the complement of a closed subset in a compact manifold.
In this situation,
the problem is referred to as the singular Yamabe problem.
If the ambient manifold is the round sphere,
this problem has been studied by several authors%
~\cites{Schoen1989,Mazzeo-Smale1991,Bettiol-Piccione-Santaro2016,Bettiol-Piccione2018}.
Related to our result,
in joint work with Andrade, Piccione, and Wei~\cite{Andrade-Case-Piccione-Wei2025-preprint},
the first author proved that
$S^{n} \setminus S^{k}$ admits infinitely many non-homothetic periodic conformally round metrics
with positive constant scalar curvature provided that $n \geq 2 k + 3$.

Based on the analogy between conformal and CR geometry,
we consider the existence problem for complete contact forms with constant Tanaka--Webster scalar curvature.
We first recall the status of this problem
on compact, connected, strictly pseudoconvex CR manifolds $(M, T^{1, 0} M)$ of dimension $2 n + 1$.
The \emph{CR Yamabe functional} $\frakF$ is defined by
\begin{equation}
\label{eq:CR-Yamabe-functional}
	\frakF(\theta)
	\coloneqq \Vol_{\theta}(M)^{- n / (n + 1)} \int_{M} R_{\theta} \, \vol_{\theta},
\end{equation}
where $\theta$ is a contact form on $M$,
$R_{\theta}$ is the Tanaka--Webster scalar curvature associated with $\theta$,
$\vol_{\theta} \coloneqq \theta \wedge (d \theta)^{n}$,
and $\Vol_{\theta}(M) \coloneqq \int_{M} \vol_{\theta}$.
The infimum of this functional is called the \emph{CR Yamabe constant} $Y(M, T^{1, 0} M)$.
A critical point of this functional corresponds to a contact form with constant Tanaka--Webster scalar curvature.
Jerison and Lee~\cites{Jerison-Lee1987,Jerison-Lee1988,Jerison-Lee1989} proved that
this functional admits a minimizer if $n \geq 2$ and $(M, T^{1, 0} M)$ is not spherical,
or if $(M, T^{1, 0} M)$ is CR diffeomorphic to the CR sphere.
Gamara and Yacoub~\cites{Gamara-Yacoub2001,Gamara2001} established the existence of a critical point
in the remaining case via the method of critical points at infinity.
The contact form constructed by their method is not necessarily a CR Yamabe minimizer.
In fact,
Cheng, Malchiodi, and Yang~\cite{Cheng-Malchiodi-Yang2023} showed that the Rossi sphere,
which is an example of a non-embeddable CR three-manifold,
admits no CR Yamabe minimizers.
We also remark that $(M, T^{1, 0} M)$ admits a CR Yamabe minimizer
if $n = 1$ and $(M, T^{1, 0} M)$ is embeddable~\cites{Cheng-Malchiodi-Yang2017,Takeuchi2020-Paneitz}.

In this paper,
we consider the non-compact case.
Let $(M, T^{1, 0} M)$ be a (possibly non-compact) connected strictly pseudoconvex CR manifold.
Denote by $\CSC(M, T^{1, 0} M)$ the set of complete contact forms on $M$
with constant Tanaka--Webster scalar curvature.
We say that two contact forms $\theta_{1}$ and $\theta_{2}$ on $M$ are \emph{homothetic}
if there exist a CR diffeomorphism $\varphi$ and a constant $c > 0$
such that $\varphi^{\ast} \theta_{1} = c \theta_{2}$;
otherwise they are said to be \emph{non-homothetic}.
The \emph{geometric moduli space} of complete contact forms
with constant Tanaka--Webster scalar curvature on $(M, T^{1, 0} M)$ is the quotient
\begin{equation}
	\GModuli(M, T^{1, 0} M)
	\coloneqq \CSC(M, T^{1, 0} M) / \sim,
\end{equation}
where $\theta_{1} \sim \theta_{2}$ if $\theta_{1}$ and $\theta_{2}$ are homothetic.
Jerison and Lee proved that $\# \GModuli(M, T^{1, 0} M) = 1$
if $M$ is compact and $Y(M, T^{1, 0} M) \leq 0$~\cite{Jerison-Lee1987}*{Theorem 7.1}
or $(M, T^{1, 0} M) = (S^{2 n + 1}, T^{1, 0} S^{2 n + 1})$~\cite{Jerison-Lee1988}*{Theorem A}.
It is also known that $\# \GModuli(M, T^{1, 0} M) = 1$
if $M$ is compact and admits an Einstein contact form~\cite{Wang2015}*{Theorem 4}.

A typical non-compact situation arises when one considers the complement of a closed subset
in a compact strictly pseudoconvex CR manifold.
Let $(M, T^{1, 0} M)$ be a compact, connected, strictly pseudoconvex CR manifold of dimension $2 n + 1$,
let $\theta$ be a contact form on $M$,
and let $\Lambda$ be a closed subset of $M$.
One seeks complete contact forms with constant Tanaka--Webster scalar curvature on $M \setminus \Lambda$;
we abbreviate the CR structure $T^{1, 0} (M \setminus \Lambda)$ for simplicity.
This leads to the \emph{singular CR Yamabe equation}
\begin{equation}
\label{eq:singular-CR-Yamabe-problem}
	\begin{cases}
		(2 + 2 / n) \Delta_{b} u + R_{\theta} u = \lambda u^{1 + 2 / n} \text{\ on $M \setminus \Lambda$}, \\
		\lim_{x \to \Lambda} u(x) = \infty,
	\end{cases}
\end{equation}
for some constant $\lambda \in \bbR$ and $0 < u \in C^{\infty}(M \setminus \Lambda)$.
The condition $\lim_{x \to \Lambda} u(x) = \infty$ is necessary for the contact form
$\whxth = u^{2 / n} \theta$ to be complete on $M \setminus \Lambda$.
Chanillo and Yang~\cite{Chanillo-Yang2025-preprint}*{Theorem 2.1} proved that $\dim_{H} \Lambda \leq n$
if there exists a solution of \cref{eq:singular-CR-Yamabe-problem} with $\lambda \geq 0$,
where $\dim_{H}$ is the Hausdorff dimension with respect to the Carnot--Carath\'{e}odory distance;
see \cref{subsection:CR-manifolds} for the definition.
Regarding existence results,
Nayatani~\cite{Nayatani1999}*{p.\ 221} constructed a complete contact form $\theta^{\prime}$
on $S^{2 n + 1} \setminus S_{\bbC}^{2 k + 1}$ with $R_{\theta'} = (n + 1) (n - 2 k - 2)$,
where $S_{\bbC}^{2 k + 1}$ is the $\bbC$-sphere
\begin{equation}
	S_{\bbC}^{2 k + 1}
	\coloneqq \Set{(z^{1}, \dots , z^{k + 1}, 0, \dots , 0) \in S^{2 n + 1} | z^{i} \in \bbC}.
\end{equation}
Guidi, Maalaoui, and Martino~\cite{Guidi-Maalaoui-Martino2024}*{Theorem 1.1}
constructed the same contact form by a different method.
They also constructed an infinite family of complete contact forms
with constant Tanaka--Webster scalar curvature on $S^{2 n + 1} \setminus S_{\bbC}^{1}$
by using a bifurcation method~\cite{Guidi-Maalaoui-Martino2024}*{Theorem 1.3}.
Furthermore,
they proved that $S^{2 n + 1} \setminus \tau(S_{\bbC}^{1})$ admits a complete contact form
with constant Tanaka--Webster scalar curvature
if $\tau \in \Diff(S^{2 n + 1})$ satisfies certain conditions \cite{Guidi-Maalaoui-Martino2024}*{Theorem 1.2}.
Afeltra~\cite{Afeltra2020}*{Theorem 1.1} constructed infinitely many $\bbZ$-periodic
contact forms on $S^{2 n + 1} \setminus \{p, - p\}$ via the Lyapunov--Schmidt method.

In this paper,
we establish the existence of infinitely many pairwise
non-homothetic complete contact forms with constant Tanaka--Webster scalar curvature
on certain non-compact strictly pseudoconvex CR manifolds.
In particular,
we prove that the geometric moduli space is infinite
on complements of $\bbR$- and $\bbC$-spheres in the CR sphere.
From this perspective,
our results can be viewed as CR analogues of the existence of
infinitely many non-homothetic complete constant scalar curvature metrics
in the non-compact setting of conformal geometry.

To explain our result in the general setting,
we start with a compact, connected, strictly pseudoconvex CR manifold $(M, T^{1, 0} M)$ of dimension $2 n + 1$
satisfying the following assumption:

\begin{keyassumption}{$(\blacklozenge)$}
\label{keyassumption}
	Assume that the CR Yamabe constant of $(M, T^{1, 0} M)$ is positive.
	If $\dim M = 3$,
	then we further assume that $(M, T^{1, 0} M)$ is universally embeddable;
	see \cref{subsection:CR-manifolds} for definition.
	If $\dim M = 5$ and $(M, T^{1, 0} M)$ is spherical,
	then we additionally assume that the developing map $\Phi \colon \wtM \to S^{5}$ is injective.
\end{keyassumption}

These additional assumptions in low dimensions are imposed
to guarantee the existence of CR Yamabe minimizers on finite covers of $M$.
If the fundamental group $\pi_{1}(M)$ of $M$ has infinitely many normal subgroups of finite index,
then there exist infinitely many pairwise non-homothetic periodic contact forms on $(\wtM, T^{1, 0} \wtM)$.

\begin{theorem}
\label{thm:infinite-geometric-moduli-general}
	Let $(M, T^{1, 0} M)$ be a compact, connected, strictly pseudoconvex CR manifold of dimension $2 n + 1$
	such that \Cref{keyassumption} holds
	and $\pi_{1}(M)$ has infinite profinite completion.
	Then $\# \GModuli(\wtM, T^{1, 0} \wtM) = \infty$.
\end{theorem}

The novelty of our work lies in the fact that
our results do not follow from existing constructions in the literature
and rely on a different geometric mechanism.
In particular,
our approach applies to examples that are not accessible by bifurcation or ODE-based methods,
which have played a central role in previous studies.

Our proof is similar to that of \cite{Andrade-Case-Piccione-Wei2025-preprint}*{Theorem 3.3}.
Since $\pi_{1}(M)$ has infinite profinite completion,
we can construct a tower of finite connected coverings over $M$.
\Cref{keyassumption} implies that each such cover has a positive CR Yamabe minimizer.
Its lift to the universal cover defines a periodic contact form
with positive constant Tanaka--Webster scalar curvature.
Suppose to the contrary that $\# \GModuli(\wtM, T^{1, 0} \wtM) < \infty$.
Then the CR diffeomorphism group $\Aut_{\CR}(\wtM, T^{1, 0} \wtM)$ does not act properly.
By a result of Schoen~\cite{Schoen1995},
it follows that $(\wtM, T^{1, 0} \wtM)$ is CR diffeomorphic to the Heisenberg group.
We conclude from the characterization of CR automorphisms of the Heisenberg group that
$Y(M, T^{1, 0} M) = 0$, a contradiction.

As an application,
consider the complements of an $\bbR$-sphere or a $\bbC$-sphere in $S^{2 n + 1}$,
where the $\bbR$-sphere
\begin{equation}
	S_{\bbR}^{k}
	\coloneqq \Set{(x^{1}, \dots , x^{k + 1}, 0, \dots , 0) \in S^{2 n + 1} | x^{i} \in \bbR}
\end{equation}
is the intersection of the totally real linear subspace $\bbR^{k + 1} \times \{0\}$ with $S^{2 n + 1}$.
There exists a torsion-free discrete subgroup $\Gamma$ of $\Aut_{\CR}(S^{2 n + 1} \setminus S_{\bbR}^{k})$
such that the quotient of $S^{2 n + 1} \setminus S_{\bbR}^{k}$ by $\Gamma$
satisfies the assumptions of \cref{thm:infinite-geometric-moduli-general}.

\begin{theorem}
\label{thm:infinite-geometric-moduli-real-hyperbolic}
	The complement $S^{2 n + 1} \setminus S_{\bbR}^{k}$ satisfies
	$\# \GModuli(S^{2 n + 1} \setminus S_{\bbR}^{k}) = \infty$ if $0 \leq k \leq n - 1$.
\end{theorem}

Similarly,
we consider the $\bbC$-sphere $S_{\bbC}^{2 k + 1}$,
which is the intersection of the complex linear subspace $\bbC^{k + 1} \times \{0\}$ with $S^{2 n + 1}$.
A similar argument to that for the $\bbR$-sphere implies the following result.

\begin{theorem}
\label{thm:infinite-geometric-moduli-complex-hyperbolic}
	The complement $S^{2 n + 1} \setminus S_{\bbC}^{2 k + 1}$ satisfies
	$\# \GModuli(S^{2 n + 1} \setminus S_{\bbC}^{2 k + 1}) = \infty$ if $0 \leq k < (n - 2) / 2$.
\end{theorem}

We add some remarks on these theorems.
Let $V$ be a real linear subspace in $\bbC^{n + 1}$
and set $\Lambda \coloneqq V \cap S^{2 n + 1} \subset S^{2 n + 1}$.
If $V$ is totally real or complex and $\dim_{H} \Lambda < n$,
then we proved that $\# \GModuli(S^{2 n + 1} \setminus \Lambda) = \infty$.
The key point of the proof is that
there exists a torsion-free convex cocompact discrete subgroup $\Gamma$ of $PU(n + 1, 1)$
whose limit set coincides with $\Lambda$.
Then the quotient of $S^{2 n + 1} \setminus \Lambda$ by $\Gamma$ satisfies \Cref{keyassumption}.
On the other hand,
when $V$ is neither totally real nor complex,
there exist no discrete subgroups in $PU(n + 1, 1)$
such that its limit set coincides with $\Lambda$,
and hence our argument does not apply.
As far as we know,
there are no existence results for complete contact forms
with constant Tanaka--Webster scalar curvature on $S^{2 n + 1} \setminus \Lambda$

We will also apply \cref{thm:infinite-geometric-moduli-general}
to circle bundles over compact \Kahler manifolds (\cref{thm:infinite-geometric-moduli-circle-bundle})
and the boundary of a Reinhardt domain (\cref{subsection:boundary-of-Reinhardt-domain}),
which is neither spherical nor Sasakian.

This paper is organized as follows.
In \cref{section:CR-geometry},
we review basic materials in CR geometry.
In \cref{section:CR-Yamabe-problem},
we recall the CR Yamabe problem and summarize known existence results that motivate the present work.
In \cref{section:complex-hyperbolic-geometry},
we collect several facts from complex hyperbolic geometry
that are needed to analyze complements of $\bbR$- and $\bbC$-spheres in the CR sphere.
In \cref{section:infinite-tower},
we study an existence problem for infinite towers of finite connected coverings;
this constitutes the technical core of the paper.
In \cref{section:proof-of-main-result},
we use this framework to prove \cref{thm:infinite-geometric-moduli-general}.
In \cref{section:complements-of-spheres},
we apply this theorem to complements of $\bbR$- and $\bbC$-spheres in the CR sphere.
Finally, in \cref{section:other-applications},
we discuss further applications to circle bundles over compact \Kahler manifolds
and to the boundary of a Reinhardt domain.

\section{CR geometry}
\label{section:CR-geometry}

\subsection{CR manifolds}
\label{subsection:CR-manifolds}

Let $M$ be a smooth $(2 n + 1)$-dimensional manifold.
A \emph{CR structure} is a rank $n$ complex subbundle $T^{1, 0} M$
of the complexified tangent bundle $T M \otimes \bbC$ such that
\begin{equation}
	T^{1, 0} M \cap T^{0, 1} M = 0, \qquad
	\comm{\Gamma(T^{1, 0} M)}{\Gamma(T^{1, 0} M)} \subset \Gamma(T^{1, 0} M),
\end{equation}
where $T^{0, 1} M$ is the complex conjugate of $T^{1, 0} M$ in $T M \otimes \bbC$.
Set $H M = \Re T^{1, 0} M$
and let $J \colon H M \to H M$ be the unique complex structure on $H M$ 
such that
\begin{equation}
	T^{1, 0} M = \Ker(J - \sqrt{- 1} \colon H M \otimes \bbC \to H M \otimes \bbC).
\end{equation}
A typical example of a CR manifold is a real hypersurface $M$ in an $(n + 1)$-dimensional complex manifold $X$;
this $M$ has the induced CR structure
\begin{equation}
	T^{1, 0} M
	\coloneqq T^{1, 0} X |_{M} \cap (T M \otimes \bbC).
\end{equation}
In particular,
the unit sphere
\begin{equation}
	S^{2 n + 1}
	\coloneqq \Set{z \in \bbC^{n + 1} | \abs{z}^{2} = 1}
\end{equation}
has the \emph{canonical} CR structure $T^{1, 0} S^{2 n + 1}$.

Let $(M, T^{1, 0} M)$ and $(M^{\prime}, T^{1, 0} M^{\prime})$ be CR manifolds.
A smooth map $f \colon M \to M^{\prime}$ is called a \emph{CR map}
if $f_{\ast} (T^{1, 0} M) \subset T^{1, 0} M^{\prime}$.
If $f$ is a diffeomorphism,
then we call $f$ a \emph{CR diffeomorphism}.
A CR manifold $(M, T^{1, 0} M)$ is said to be \emph{spherical}
if it is locally isomorphic to $(S^{2 n + 1}, T^{1, 0} S^{2 n + 1})$.

A CR structure $T^{1, 0} M$ is said to be \emph{strictly pseudoconvex}
if there exists a smooth real-valued one-form $\theta$ annihilating exactly $H M$
such that the \emph{Levi form}
\begin{equation}
	L_{\theta}(V, W)
	\coloneqq d \theta(V, J W)
	\qquad (V, W \in H M)
\end{equation}
is positive definite.
We call such a one-form a \emph{contact form}.
The triple $(M, T^{1, 0} M, \theta)$ is called a \emph{pseudo-Hermitian manifold}.
Denote by $T$ the \emph{Reeb vector field} with respect to $\theta$; 
that is,
the unique vector field satisfying
\begin{equation}
	\theta(T) = 1, \qquad T \contr d \theta = 0.
\end{equation}
Let $(Z_{\alpha})$ be a local frame of $T^{1, 0} M$,
and set $Z_{\ovxa} = \overline{Z_{\alpha}}$.
Then
$(T, Z_{\alpha}, Z_{\ovxa})$ gives a local frame of $T M \otimes \bbC$,
called an \emph{admissible frame}.
Its dual frame $(\theta, \theta^{\alpha}, \theta^{\ovxa})$
is called an \emph{admissible coframe}.
The two-form $d \theta$ is written as
\begin{equation}
	d \theta
	= \sqrt{- 1} h_{\alpha \ovxb} \theta^{\alpha} \wedge \theta^{\ovxb},
\end{equation}
where $(h_{\alpha \ovxb})$ is a positive definite Hermitian matrix.
We use $h_{\alpha \ovxb}$ and its inverse $h^{\alpha \ovxb}$
to raise and lower indices of tensors.

The flat model of pseudo-Hermitian manifolds is the \emph{Heisenberg group} $\bbH^{2 n + 1}$,
that is,
the Lie group with the underlying manifold $\bbR \times \bbC^{n}$ and the multiplication
\begin{equation}
	(t, z) \cdot (t^{\prime}, z^{\prime})
	\coloneqq (t + t^{\prime} + 2 \Im (z \cdot \ovz^{\prime}), z + z^{\prime}).
\end{equation}
For $\alpha = 1, \dots , n$,
we introduce a left-invariant complex vector field $Z_{\alpha}$ by
\begin{equation}
	Z_{\alpha}
	\coloneqq \pdv{}{z^{\alpha}} + \sqrt{- 1} \ovz^{\alpha} \pdv{}{t}.
\end{equation}
The canonical CR structure $T^{1, 0} \bbH^{2 n + 1}$
is spanned by $Z_{1}, \dots , Z_{n}$.
Define a left-invariant one-form $\theta$ on $\bbH^{2 n + 1}$ by
\begin{equation}
	\theta
	\coloneqq d t + \sqrt{-1} \sum_{\alpha = 1}^{n}
		(z^{\alpha} d \ovz^{\alpha} - \ovz^{\alpha} d z^{\alpha}).
\end{equation}
Then $\theta$ annihilates $T^{1, 0} \bbH^{2 n + 1}$
and the Levi form $L_{\theta}$ satisfies
$L_{\theta}(Z_{\alpha}, Z_{\ovxb}) = 2 \delta_{\alpha \beta}$;
in particular,
$(\bbH^{2 n + 1}, T^{1, 0} \bbH^{2 n + 1})$ is a strictly pseudoconvex CR manifold
and $\theta$ is a contact form on $\bbH^{2 n + 1}$.

A CR manifold $(M, T^{1, 0} M)$ is said to be \emph{embeddable}
if there exists a smooth embedding $F \colon M \to \bbC^{N}$
such that $F_{\ast}(T^{1, 0} M) \subset T^{1, 0} \bbC^{N}$.
It is known that a compact strictly pseudoconvex CR manifold $(M, T^{1, 0} M)$
of dimension at least five must be embeddable~\cite{Boutet_de_Monvel1975}.
On the other hand,
there are many non-embeddable compact strictly pseudoconvex CR manifolds of dimension three;
see for example \cite{Burns-Epstein1990-Embed}.
We say that $(M, T^{1, 0} M)$ is \emph{universally embeddable}
if every finite cover of $(M, T^{1, 0} M)$ is embeddable.
This condition is strictly stronger than the embeddability condition~\cite{Case-Yang2025}.

Assume that $M$ is connected and fix a contact form $\theta$ on $M$.
We can also endow $M$ with the \emph{Carnot--Carath\'{e}odory distance} $d_{\CC}$ on $M$ as follows.
For any $p, q \in M$,
we can find a smooth path $c \colon \clcl{0}{1} \to M$
such that $c(0) = p$, $c(1) = q$, and $\theta(c^{\prime}(t)) = 0$.
Then $d_{\CC}(p, q)$ is the infimum of the length of such curves.
Denote by $\dim_{H} A$ the Hausdorff dimension of $A \subset M$ with respect to $d_{\CC}$.
The Carnot--Carath\'{e}odory distance $d_{\CC}$ induces the standard topology of $M$,
but the Hausdorff dimension $\dim_{H} M$ of $M$ is equal to $2 n + 2$,
which does not coincide with its topological dimension.
See \cite{Gromov1996} for more details.

\subsection{Tanaka--Webster connection}
\label{subsection:TW-connection}

A contact form $\theta$ induces a canonical connection $\nabla$ on $T M$,
called the \emph{Tanaka--Webster connection} with respect to $\theta$.
It is defined by
\begin{equation}
	\nabla T
	= 0,
	\quad
	\nabla Z_{\alpha}
	= \omega_{\alpha} {}^{\beta} Z_{\beta},
	\quad
	\nabla Z_{\ovxa}
	= \omega_{\ovxa} {}^{\ovxb} Z_{\ovxb},
\end{equation}
where $\omega_{\ovxa} {}^{\ovxb} = \overline{\omega_{\alpha} {}^{\beta}}$,
with the following structure equations:
\begin{gather}
\label{eq:str-eq-of-TW-conn1}
	d \theta^{\beta}
	= \theta^{\alpha} \wedge \omega_{\alpha} {}^{\beta}
	+ A^{\beta} {}_{\ovxa} \theta \wedge \theta^{\ovxa}, \\
\label{eq:str-eq-of-TW-conn2}
	d h_{\alpha \ovxb}
	= \omega_{\alpha} {}^{\gamma} h_{\gamma \ovxb}
		+ h_{\alpha \ovxg} \omega_{\ovxb} {}^{\ovxg}.
\end{gather}
The tensor $A_{\alpha \beta} = \overline{A_{\ovxa \ovxb}}$
is symmetric and is called the \emph{Tanaka--Webster torsion}.
The \emph{sub-Laplacian} $\Delta_{b}$ is defined by
\begin{equation}
	\Delta_{b} u
	\coloneqq - \nabla^{\alpha} \nabla_{\alpha} u - \nabla^{\ovxb} \nabla_{\ovxb} u.
\end{equation}

The curvature form
$\Omega_{\alpha} {}^{\beta} \coloneqq d \omega_{\alpha} {}^{\beta}
- \omega_{\alpha} {}^{\gamma} \wedge \omega_{\gamma} {}^{\beta}$
of the Tanaka--Webster connection satisfies
\begin{equation}
\label{eq:curvature-form-of-TW-connection}
	\Omega_{\alpha} {}^{\beta}
	= R_{\alpha} {}^{\beta} {}_{\rho \ovxs} \theta^{\rho} \wedge \theta^{\ovxs}
		\qquad \text{modulo $\theta, \theta^{\rho} \wedge \theta^{\gamma},
		\theta^{\ovxg} \wedge \theta^{\ovxs}$}.
\end{equation}
We call the tensor $R_{\alpha} {}^{\beta} {}_{\rho \ovxs}$
the \emph{Tanaka--Webster curvature}.
This tensor has the symmetry 
\begin{equation}
	R_{\alpha \ovxb \rho \ovxs}
	= R_{\rho \ovxb \alpha \ovxs}
	= R_{\alpha \ovxs \rho \ovxb}.
\end{equation}
Contraction of indices gives the \emph{Tanaka--Webster Ricci curvature}
$\Ric_{\rho \ovxs} = R_{\alpha} {}^{\alpha} {}_{\rho \ovxs}$
and the \emph{Tanaka--Webster scalar curvature}
$R_{\theta} = \Ric_{\rho} {}^{\rho}$.
When $n \geq 2$,
the \emph{Chern tensor} $S_{\alpha \ovxb \rho \ovxs}$ is the completely trace-free part
of $R_{\alpha \ovxb \rho \ovxs}$;
this tensor is the CR analogue of the Weyl tensor in conformal geometry.
It is known that the Chern tensor vanishes identically
if and only if $(M, T^{1, 0} M)$ is spherical~\cite{Chern-Moser1974}.

Fix a contact form $\theta$ on $M$.
Then any contact form $\whxth$ is of the form $u^{2 / n} \theta$,
where
\begin{equation}
	u \in C^{\infty}_{+}(M) \coloneqq \Set{u \in C^{\infty}(M) | u > 0}.
\end{equation}
Under this conformal change,
the Tanaka--Webster scalar curvature $R_{\wtxth}$ with respect to $\whxth$ satisfies
\begin{equation}
	R_{\wtxth}
	= u^{- 1 - 2 / n} \sbra*{(2 + 2 / n) \Delta_{b} u + R_{\theta} u};
\end{equation}
see \cite{Jerison-Lee1987}*{p.\ 174} for example.

\subsection{Sasakian manifolds}
\label{subsection:Sasakian-manifolds}

Sasakian manifolds constitute an important class of pseudo-Hermitian manifolds.
See~\cite{Boyer-Galicki2008} for a comprehensive introduction.

A \emph{Sasakian manifold} is a pseudo-Hermitian manifold $(S, T^{1, 0}S, \theta)$
with vanishing Tanaka--Webster torsion.
This condition is equivalent to the requirement that the Reeb vector field $T$ of $\theta$
preserves the CR structure $T^{1, 0} S$.

A typical example of a Sasakian manifold
is the circle bundle associated with a negative Hermitian holomorphic line bundle.
Let $Y$ be an $n$-dimensional complex manifold
and $(L, h)$ a Hermitian holomorphic line bundle over $Y$
such that
\begin{equation}
	\omega = - \sqrt{- 1} \Theta_{h} = d d^{c} \log h
\end{equation}
is a \Kahler form on $Y$,
where $d^{c} \coloneqq (\sqrt{- 1} / 2)(\delb - \partial)$.
Consider the circle bundle
\begin{equation}
	S = \Set{ v \in L | h(v, v) = 1}
\end{equation}
over $Y$,
which is a real hypersurface in the total space of $L$
and has the canonical CR structure $T^{1, 0} S$.
The one-form $\theta = d^{c} \log h |_{S}$ is a contact form on $S$
and its Tanaka--Webster scalar curvature $R_{\theta}$ coincides with $p^{\ast} R_{\omega}$,
where $p \colon S \to Y$ is the projection
and $R_{\omega}$ is the scalar curvature of $\omega$.
Moreover,
the CR manifold $(S, T^{1, 0} S)$ is spherical if and only if
$(Y, \omega)$ is Bochner-flat.
See \cite{Takeuchi2022-Chern}*{Section 2.3} for more details.

\subsection{Schoen's Theorem for CR automorphism group}

The CR sphere and the Heisenberg group play a distinguished role
in CR geometry from the viewpoint of symmetry.
To explain this,
we begin with some definitions.

\begin{definition}
	Let $(M, T^{1, 0} M)$ be a strictly pseudoconvex CR manifold of dimension $2 n + 1$.
	The \emph{CR automorphism group} $\Aut_{\CR}(M, T^{1, 0} M)$
	consists of all CR diffeomorphisms $\varphi \in \Diff(M)$
	equipped with the compact-open topology.
	We say that $\Aut_{\CR}(M, T^{1, 0} M)$ \emph{acts properly} on $M$
	if $\Set{\varphi \in \Aut_{\CR}(M, T^{1, 0} M) | \varphi(K) \cap K \neq \emptyset}$
	is relatively compact in $\Aut_{\CR}(M, T^{1, 0} M)$ for any compact set $K \subset M$.
\end{definition}

Schoen's theorem for CR automorphism group states that
$\Aut_{\CR}(M, T^{1, 0} M)$ acts properly except for
$(S^{2 n + 1}, T^{1, 0} S^{2 n + 1})$ or $(\bbH^{2 n + 1}, T^{1, 0} \bbH^{2 n + 1})$.

\begin{theorem}[\cite{Schoen1995}*{Theorems 3.3' and 3.4'}]
\label{thm:Schoen's-theorem}
	Let $(M, T^{1, 0} M)$ be a strictly pseudoconvex CR manifold of dimension $2 n + 1$.
	If the CR automorphism group $\Aut_{\CR}(M, T^{1, 0} M)$ of $(M, T^{1, 0} M)$ does not act properly on $M$,
	then $(M, T^{1, 0} M)$ is CR diffeomorphic to
	$(S^{2 n + 1}, T^{1, 0} S^{2 n + 1})$ or $(\bbH^{2 n + 1}, T^{1, 0} \bbH^{2 n + 1})$.
\end{theorem}

\subsection{The CR automorphism group of $S^{2 n + 1}$ and $\bbH^{2 n + 1}$}

The $(n + 1)$-dimensional \emph{complex hyperbolic space} is the ball
\begin{equation}
	\bbB_{\bbC}^{n + 1}
	\coloneqq \Set{z = (z^{1}, \dots , z^{n + 1}) \in \bbC^{n + 1} | \abs{z}^{2}  < 1}
\end{equation}
endowed with the complete \Kahler--Einstein form
\begin{equation}
	\omega_{\bbB}
	\coloneqq - \frac{1}{2} d d^{c} \log (1 - \abs{z}^{2}).
\end{equation}
Denote by $d(z, w)$ the geodesic distance
between $z \in \bbB_{\bbC}^{n + 1}$ and $w \in \bbB_{\bbC}^{n + 1}$.
\emph{Complex geodesics} are the non-empty intersections of complex lines and $\bbB_{\bbC}^{n + 1}$.
The boundary $S^{2 n + 1}$ of $\bbB_{\bbC}^{n + 1}$
has the canonical CR structure $T^{1, 0} S^{2 n + 1}$ described in \cref{subsection:CR-manifolds}.
The boundary of a complex geodesic is a circle in $S^{2 n + 1}$
that is transverse to the canonical contact structure on $S^{2 n + 1}$,
which is known as a \emph{chain};
see \cite{Chern-Moser1974} for a more general definition.

Let $U(n + 1, 1)$ be the unitary group with respect to
the Hermitian form determined by $\diag(- 1, 1, \dots , 1)$;
that is,
\begin{equation}
	U(n + 1, 1)
	\coloneqq \Set{A \in GL(n + 2, \bbC) | A^{\ast} \Hermform A = \Hermform}.
\end{equation}
This group acts on both $\bbB_{\bbC}^{n + 1}$ and $S^{2 n + 1}$ by the fractional linear transformation
\begin{equation}
\label{eq:group-action}
	\begin{pmatrix}
		a & b \\
		c & D
	\end{pmatrix}
	\cdot z
	\coloneqq \frac{c + D z}{a + b z}.
\end{equation}
This action preserves the \Kahler form $\omega_{\bbB}$ and the CR structure $T^{1, 0} S^{2 n + 1}$.
The map \cref{eq:group-action} is equal to the identity map if and only if
the matrix is proportional to the identity matrix.
Hence the action of $U(n + 1, 1)$ descends to that of the projective unitary group $PU(n + 1, 1)$.
Moreover,
it is known that
\begin{equation}
	\Aut(\bbB_{\bbC}^{n + 1}, \omega_{\bbB})
	= \Aut_{\CR}(S^{2 n + 1}, T^{1, 0} S^{2 n + 1})
	= PU(n + 1, 1);
\end{equation}
see~\cite{Burns-Shnider1976} for example.

Under the Cayley transform,
the Heisenberg group can be identified with $S^{2 n + 1}$ with the point at infinity removed.
In particular,
the CR automorphism group of the Heisenberg group is isomorphic to
that of the sphere fixing the point at infinity.
More precisely,
\begin{equation}
	\Aut_{\CR}(\bbH^{2 n + 1}, T^{1, 0} \bbH^{2 n + 1})
	= \bbH^{2 n + 1} \rtimes CU(n),
\end{equation}
where $CU(n)$ is the conformal unitary group.

\begin{lemma}
\label{lem:CR-Yamabe-constant-Heisenberg-nilmanifold}
	Let $(M, T^{1, 0} M)$ be a compact, connected, strictly pseudoconvex CR manifold of dimension $2 n + 1$.
	If its universal cover is CR diffeomorphic to the Heisenberg group,
	then $Y(M, T^{1, 0} M) = 0$.
\end{lemma}

\begin{proof}
	The fundamental group $\pi_{1}(M)$ can be identified with a discrete subgroup
	of the CR automorphism group of the Heisenberg group.
	Since $\pi_{1}(M)$ acts properly discontinuously on the Heisenberg group,
	this is a subgroup of $\bbH^{2 n + 1} \rtimes U(n)$~\cite{Burns-Shnider1976}*{Proposition 5.6}.
	In particular,
	$\pi_{1}(M)$ preserves the standard contact form on $\bbH^{2 n + 1}$,
	and so $(M, T^{1, 0} M)$ admits a contact form with vanishing Tanaka--Webster curvature.
	This implies $Y(M, T^{1, 0} M) = 0$.
\end{proof}

\section{CR Yamabe problem}
\label{section:CR-Yamabe-problem}

Let $(M, T^{1, 0} M)$ be a compact, connected, strictly pseudoconvex CR manifold of dimension $2 n + 1$.
Consider the CR Yamabe functional $\frakF$ given by \cref{eq:CR-Yamabe-functional}.
Note that $\frakF(c \theta) = \frakF(\theta)$ for any constant $c > 0$.
As with the Yamabe problem,
the functional $\frakF(\theta)$ can be written as a non-linear Rayleigh quotient
\begin{equation}
	\frakR_{\theta}(u)
	\coloneqq \frakF(u^{2 / n} \theta)
	= \frac{\int_{M} \sbra*{(2 + 2 / n) \abs{d u}_{\theta}^{2} + R_{\theta} u^{2}} \vol_{\theta}}
		{(\int_{M} u^{2 + 2 / n} \, \vol_{\theta})^{n / (n + 1)}},
\end{equation}
where $\abs{d u}_{\theta}^{2} = L_{\theta}^{\ast}(d u|_{H M}, d u|_{H M})$~\cite{Jerison-Lee1987}.
The infimum of $\frakR_{\theta}$ over $C^{\infty}_{+}(M)$ defines the \emph{CR Yamabe constant}
\begin{equation}
	Y(M, T^{1, 0} M)
	\coloneqq \inf_{u \in C^{\infty}_{+}(M)} \frakR_{\theta}(u),
\end{equation}
which is a CR invariant of $(M, T^{1, 0} M)$.
Analogous to the conformal case,
one has
\begin{equation}
	Y(M, T^{1, 0} M)
	\leq Y(S^{2 n + 1}, T^{1, 0} S^{2 n + 1})
	= 2 n (n + 1) \pi
	\eqqcolon Y_{n};
\end{equation}
see \cite{Jerison-Lee1987}*{Theorem 3.4(b)} and \cite{Jerison-Lee1988}*{Corollary C}.
Moreover,
the functional $\frakR_{\theta}$ has a minimizer
if $Y(M, T^{1, 0} M) < Y_{n}$~\cite{Jerison-Lee1987}*{Theorem 3.4(c)}
or $(M, T^{1, 0} M) = (S^{2 n + 1}, T^{1, 0} S^{2 n + 1})$~\cite{Jerison-Lee1988}*{Theorem A}.
It is natural to ask whether $Y(M, T^{1, 0} M) < Y_{n}$ holds
for compact strictly pseudoconvex CR manifolds not CR diffeomorphic to $(S^{2 n + 1}, T^{1, 0} S^{2 n + 1})$.
Jerison and Lee~\cite{Jerison-Lee1989}*{Theorem A} proved that
we have $Y(M, T^{1, 0} M) < Y_{n}$
if $n \geq 2$ and $(M, T^{1, 0} M)$ is not spherical.
If $(M, T^{1, 0} M)$ is spherical and $Y(M, T^{1, 0} M) > 0$,
Cheng, Chiu, and Yang~\cite{Cheng-Chiu-Yang2014} showed that
$Y(M, T^{1, 0} M) < Y_{n}$ still holds
under the assumption that the developing map $\Phi \colon \wtM \to S^{2 n + 1}$ is injective,
which automatically holds for the $n \geq 3$ case.
If $n = 1$,
combining \cite{Cheng-Malchiodi-Yang2017}*{Theorem 1.1}
and \cite{Takeuchi2020-Paneitz}*{Theorem 1.1} implies that
$Y(M, T^{1, 0} M) < Y_{1}$ if $(M, T^{1, 0} M)$ is embeddable.
Note that the Rossi sphere,
an example of a non-embeddable CR manifold,
has no minimizers of the functional $\frakR_{\theta}$~\cite{Cheng-Malchiodi-Yang2023}*{Theorem 1.3}.

\section{Complex hyperbolic geomerty}
\label{section:complex-hyperbolic-geometry}

In this section,
we recall some basic facts on complex hyperbolic geometry;
see \cites{Goldman1999,Corlette-Iozzi1999,Kapovich2022} and the references therein for more details.

\subsection{Complex hyperbolic Kleinian group}

Let $\Gamma$ be a discrete subgroup of $PU(n + 1, 1)$,
which is known as a \emph{complex hyperbolic Kleinian group}.
Note that $\Gamma$ is discrete if and only if
$\Gamma$ acts properly discontinuously on $\bbB_{\bbC}^{n + 1}$.
We write $X_{\Gamma}$ for the quotient of $\bbB_{\bbC}^{n + 1}$ by $\Gamma$.
The action of $\Gamma$ on $\bbB_{\bbC}^{n + 1}$ is free
if and only if $\Gamma$ is torsion-free;
in this case,
$X_{\Gamma}$ is a smooth complex manifold.

The \emph{limit set} $\Lambda_{\Gamma}$ of a discrete subgroup $\Gamma$ of $PU(n + 1, 1)$
is the set of accumulation points in $\overline{\bbB}_{\bbC}^{n + 1}$
of the $\Gamma$-orbit of any point in $\bbB_{\bbC}^{n + 1}$,
which is a closed subset of $S^{2 n + 1}$.
We call $\Gamma$ \emph{elementary} if $\# \Lambda_{\Gamma} \leq 2$;
otherwise we call $\Gamma$ \emph{non-elementary}.
If $\Gamma$ is non-elementary,
then $\Lambda_{\Gamma}$ is the smallest non-empty closed $\Gamma$-invariant subset of $S^{2 n + 1}$;
see \cite{Kapovich2022}*{Proposition 2} for example.
The complement $\Omega_{\Gamma} \coloneqq S^{2 n + 1} \setminus \Lambda_{\Gamma}$ of $\Lambda_{\Gamma}$
is called the \emph{domain of discontinuity}.
This is the largest open subset of $S^{2 n + 1}$
on which $\Gamma$ acts properly discontinuously.
If $\Omega_{\Gamma}$ is non-empty,
denote by $M_{\Gamma}$ (resp.\ $\ovX_{\Gamma}$)
the quotient of $\Omega_{\Gamma}$ (resp.\ $\bbB_{\bbC}^{n + 1} \cup \Omega_{\Gamma}$) by $\Gamma$.
The action of $\Gamma$ on $\Omega_{\Gamma}$ is free
if and only if $\Gamma$ is torsion-free;
in this case,
$\ovX_{\Gamma}$ is a smooth complex manifold with boundary $M_{\Gamma}$,
and $M_{\Gamma}$ is a spherical CR manifold.

Let $\Gamma$ be a discrete subgroup of $PU(n + 1, 1)$ satisfying $\# \Lambda_{\Gamma} \geq 2$.
The \emph{closed convex hull} $C_{\Gamma}$ of $\Lambda_{\Gamma}$
is the intersection of all closed convex subsets in $\bbB_{\bbC}^{n + 1}$
whose boundary contains $\Lambda_{\Gamma}$.
This subset is invariant under $\Gamma$,
and we say that $\Gamma$ is \emph{convex cocompact}
if the quotient of $C_{\Gamma}$ by $\Gamma$ is compact.
This condition is equivalent to that $\ovX_{\Gamma}$ is compact~\cite{Bowditch1995}.
In particular if $\Gamma$ is torsion-free and convex cocompact,
then $(M_{\Gamma}, T^{1, 0} M_{\Gamma})$ is a compact embeddable strictly pseudoconvex CR manifold.

Take $z, w \in \bbB_{\bbC}^{n + 1}$.
For $s > 0$,
we define
\begin{equation}
	\Phi_{s}(z, w)
	\coloneqq \sum_{\gamma \in \Gamma} e^{- s d(z, \gamma w)} \in \opcl{0}{+ \infty}.
\end{equation}
The \emph{critical exponent} of $\Gamma$ is
\begin{equation}
	\delta_{\Gamma}
	\coloneqq \inf \Set{s \in \opop{0}{+ \infty} | \Phi_{s}(z, w) < + \infty};
\end{equation}
note that the condition $\Phi_{s}(z, w) < + \infty$
is independent of the choice of $z, w \in \bbB_{\bbC}^{n + 1}$.
It is known that $0 \leq \delta_{\Gamma} \leq 2 n + 2$
and $\delta_{\Gamma} = 0$ if and only if $\Gamma$ is elementary.
Moreover,
one has $\delta_{\Gamma} = \dim_{H} \Lambda_{\Gamma}$ if $\Gamma$ is convex cocompact%
~\cite{Corlette-Iozzi1999}*{Theorem 6.1}.

Nayatani~\cite{Nayatani1999}, Yue~\cite{Yue1999}, and Wang~\cite{Wang2003} have independently
constructed a $\Gamma$-invariant contact form $\theta_{\Gamma}$ on $\Omega_{\Gamma}$
by using the Patterson--Sullivan measure of $\Gamma$.
If $\Gamma$ is torsion-free,
this contact form descends to that on $M_{\Gamma}$,
which we continue to write $\theta_{\Gamma}$ by abuse of notation.

\begin{theorem}[\cite{Nayatani1999}*{Theorem 2.4(ii)}]
	If $\# \Lambda_{\Gamma} \geq 2$ and $\delta_{\Gamma} < n$,
	then the Tanaka--Webster scalar curvature of $\theta_{\Gamma}$ is positive
	unless the limit set $\Lambda_{\Gamma}$ lies properly in a chain.
	If $\Lambda_{\Gamma}$ lies properly in a chain $C$,
	then the Tanaka--Webster scalar curvature of $\theta_{\Gamma}$ is non-negative
	and vanishes precisely on $C \setminus \Lambda_{\Gamma}$.
\end{theorem}

In particular if $\Gamma$ is torsion-free convex cocompact and $\delta_{\Gamma} < n$,
then the CR Yamabe constant $Y(M_{\Gamma}, T^{1, 0} M_{\Gamma})$ must be positive.

\section{Infinite tower of finite connected coverings}
\label{section:infinite-tower}

\subsection{Profinite completion and residual finiteness}

To prove our main theorem,
we consider an infinite tower of finite connected coverings.
The existence of such a tower is equivalent to the requirement
that the fundamental group has infinitely many finite-index subgroups.
The profinite completion encodes the information of the finite-index normal subgroups.

\begin{definition}
	Let $\Gamma$ be a discrete group.
	The \emph{profinite completion} $\whxcg$ of $\Gamma$ is the inverse limit
	\begin{equation}
		\whxcg
		\coloneqq \underset{\substack{N \triangleleft \Gamma \\ [\Gamma : N] < \infty}}{\varprojlim} \Gamma / N.
	\end{equation}
	The group $\whxcg$ is a compact, Hausdorff, and totally disconnected topological group.
	It has a natural morphism $\iota_{\Gamma} \colon \Gamma \to \whxcg$ with dense image.
	We say that $\Gamma$ has \emph{infinite profinite completion}
	if $\whxcg$ is infinite,
	or equivalently,
	$\Gamma$ has infinitely many finite-index normal subgroups.
\end{definition}

The property that the profinite completion is infinite is preserved under surjective homomorphisms.

\begin{lemma}
\label{lem:sujective-infinite-profinite-completion}
	Let $\Gamma$ and $\Gamma^{\prime}$ be discrete groups
	and $\phi \colon \Gamma^{\prime} \to \Gamma$ be a surjective group homomorphism.
	If $\Gamma$ has infinite profinite completion,
	then so does $\Gamma^{\prime}$.
\end{lemma}

\begin{proof}
	The group homomorphism $\phi$ induces a continuous group homomorphism
	$\whxph \colon \whxcg^{\prime} \to \whxcg$.
	Since $\whxcg^{\prime}$ is compact and $\whxcg$ is Hausdorff,
	the image $\whxph(\whxcg^{\prime})$ is closed in $\whxcg$.
	Moreover,
	$\whxph(\whxcg^{\prime})$ contains
	\begin{equation}
		\whxph(\iota_{\Gamma^{\prime}}(\Gamma^{\prime}))
		= \iota_{\Gamma}(\phi(\Gamma^{\prime}))
		= \iota_{\Gamma}(\Gamma),
	\end{equation}
	which is dense in $\whxcg$.
	Hence $\whxph(\whxcg^{\prime})$ coincides with $\whxcg$
	and so $\whxph$ is surjective.
	It follows from $\# \whxcg = \infty$ that $\# \whxcg^{\prime} = \infty$.
\end{proof}

We next consider residual finiteness,
which means that the intersection of all finite-index normal subgroups is trivial.

\begin{definition}
	Let $\Gamma$ be a discrete group.
	It is said to be \emph{residually finite}
	if for any $\gamma \in \Gamma \setminus \{e\}$,
	there exists a normal subgroup $N \triangleleft \Gamma$ with finite index
	such that $\gamma \notin N$.
\end{definition}

This condition is equivalent to the injectivity of $\iota_{\Gamma} \colon \Gamma \to \whxcg$.
Note that $\Gamma$ has infinite profinite completion
if $\Gamma$ is infinite and residually finite.
It is known that a finitely generated linear subgroup is residually finite.

\begin{lemma}[Selberg--Mal'cev lemma~\cite{Ratcliffe2019}*{Section 7.6}]
\label{lem:Selberg-Malcev}
	If $\Gamma \subset GL(m, \bbC)$ is a finitely generated linear subgroup,
	then $\Gamma$ is residually finite.
\end{lemma}

The group $PU(n + 1, 1)$,
which can be identified with $\Aut_{\CR}(S^{2 n + 1}, T^{1, 0} S^{2 n + 1})$,
is not a subgroup of $GL(n + 2, \bbC)$.
However,
the same statement as \cref{lem:Selberg-Malcev} still holds.

\begin{proposition}
\label{prop:residually-finite-PU}
	If $\Gamma \subset PU(n + 1, 1)$ is a finitely generated discrete subgroup,
	then $\Gamma$ is residually finite.
\end{proposition}

\begin{proof}
	Let $\frakg$ be the Lie algebra of $U(n + 1, 1)$
	and $\Ad \colon U(n + 1, 1) \to GL(\frakg)$ be the adjoint representation.
	The kernel of $\Ad$ coincides with the center $Z$ of $U(n + 1, 1)$,
	and $U(n + 1, 1) / Z$ is isomorphic to $PU(n + 1, 1)$.
	Hence we obtain an injective group homomorphism $PU(n + 1, 1) \to GL(\frakg)$.
	The Selberg--Mal'cev lemma implies that any finitely generated discrete subgroup of $PU(n + 1, 1)$
	is residually finite.
\end{proof}

\subsection{Infinite tower of finite connected coverings}

Let $M$ be a compact and connected manifold.
By an \emph{infinite tower} of finite connected coverings of degree $(m_{j})_{j = 1}^{\infty}$ over $M$,
we mean a sequence $(\pi_{j} \colon M_{j} \to M_{j - 1})_{j = 1}^{\infty}$
of finite connected coverings satisfying
$M_{0} = M$ and $m_{j} \coloneqq \deg \pi_{j} \geq 2$.
For $0 \leq l \leq j$,
set
\begin{equation}
	\Pi_{j}^{l}
	\coloneqq \pi_{l + 1} \circ \dots \circ \pi_{j} \colon M_{j} \to M_{l}
\end{equation}
with convention $\Pi_{j}^{j} = \Id_{M_{j}}$.
Note that $\Pi_{j}^{l}$ is also a covering map of degree $m_{l + 1} \dotsm m_{j}$.
We say the infinite tower $(\pi_{j} \colon M_{j} \to M_{j - 1})_{j = 1}^{\infty}$ is \emph{regular}
if the covering $\Pi_{j}^{0} \colon M_{j} \to M_{0} = M$ is regular for each $j \in \bbZ_{> 0}$,
or equivalently,
$(\Pi_{j}^{0})_{\ast} \pi_{1}(M_{j})$ is a normal subgroup of $\pi_{1}(M)$.
Let $\wtM$ be the universal cover of $M$
and let $\wtxp_{j} \colon \wtM \to M_{j}$ be the universal covering map of $M_{j}$.
Note that $\wtM$ must be non-compact since $\deg \Pi_{j}^{0} \to \infty$ as $j \to \infty$.

Assume that the fundamental group $\pi_{1}(M)$ of $M$ has infinite profinite completion.
Under this assumption,
Bettiol and Piccione~\cite{Bettiol-Piccione2018}*{Section 3.3} constructed
an infinite tower of finite connected coverings over $M$ as follows.
There exists a nested sequence
\begin{equation}
	\dots \subset \Gamma_{j} \subset \dots \subset \Gamma_{1} \subset \Gamma_{0} = \pi_{1}(M)
\end{equation}
of normal subgroups of finite index at least two.
Then $M_{j} \coloneqq \wtM / \Gamma_{j}$ has the canonical finite covering
$\pi_{j} \colon M_{j} \to M_{j - 1}$ of degree $[\Gamma_{j - 1} : \Gamma_{j}]$,
which defines a regular infinite tower of finite connected coverings
of degree $(m_{j} = [\Gamma_{j - 1} : \Gamma_{j}])_{j = 1}^{\infty}$ over $M$.

\section{Proof of \cref{thm:infinite-geometric-moduli-general}}
\label{section:proof-of-main-result}

In this section,
we provide a condition,
in a general setting,
under which the geometric moduli space of complete contact forms
with constant Tanaka--Webster scalar curvature on the universal cover is infinite.

\begin{proposition}
\label{prop:existence-of-tower-implies-Heisenberg}
	Let $(M, T^{1, 0} M)$ be a compact, connected, strictly pseudoconvex CR manifold of dimension $2 n + 1$.
	Suppose there are an infinite tower $(\pi_{j} \colon M_{j} \to M_{j - 1})_{j = 1}^{\infty}$
	of finite connected coverings of degree $(m_{j})_{j = 1}^{\infty}$ over $M$
	and a sequence $(\theta_{j})_{j = 0}^{\infty}$
	of unit volume CR Yamabe minimizers on $(M_{j}, T^{1, 0} M_{j})$
	such that for each $j \in \bbZ_{> 0}$,
	there exist $\Phi_{j} \in \Aut_{\CR}(\wtM, T^{1, 0} \wtM)$ and $c_{j} \in \bbR_{> 0}$
	satisfying $\Phi_{j}^{\ast} \wtxp^{\ast} \theta_{0} = c_{j} \wtxp_{j}^{\ast} \theta_{j}$.
	Then $Y(M, T^{1, 0} M) \leq 0$.
\end{proposition}

\begin{proof}
	Suppose to the contrary that $Y(M, T^{1, 0} M) > 0$.
	Let $F \subset \wtM$ be a fundamental domain of $\pi_{1}(M)$
	and let $\Aut(\wtxp_{j})$ be the group of deck transformations of $\wtxp_{j}$.
	Note that $\Aut(\wtxp_{j}) = \pi_{1}(M_{j})$
	is a subgroup of $\Aut(\wtxp) = \pi_{1}(M)$ of index $m_{1} \dotsm m_{j}$.
	Take $\tau_{j} \in \Aut(\wtxp)$ such that
	\begin{equation}
		\Vol_{\tau_{j}^{\ast} \wtxp_{j}^{\ast} \theta_{j}}(F)
		= \min \Set{\Vol_{\sigma^{\ast} \wtxp_{j}^{\ast} \theta_{j}}(F) | \sigma \in \Aut(\wtxp)}.
	\end{equation}
	This yields $\Vol_{\tau_{j}^{\ast} \wtxp_{j}^{\ast} \theta_{j}}(F) \leq (m_{1} \dotsm m_{j})^{- 1}$.
	Choose $\sigma_{j} \in \Aut(\wtxp)$ so that
	\begin{equation}
		\Psi_{j} \coloneqq \sigma_{j} \circ \Phi_{j} \circ \tau_{j} \in \Aut_{\CR}(\wtM, T^{1, 0} \wtM)
	\end{equation}
	satisfies $\Psi_{j}(F) \cap F \neq \emptyset$.
	Consider the contact form
	\begin{equation}
		\wtxth_{j}
		\coloneqq \Psi_{j}^{\ast} \wtxp^{\ast} \theta_{0}
		= \tau_{j}^{\ast} \Phi_{j}^{\ast} \sigma_{j}^{\ast} \wtxp^{\ast} \theta_{0}
		= \tau_{j}^{\ast} \Phi_{j}^{\ast} \wtxp^{\ast} \theta_{0}
		= c_{j} \tau_{j}^{\ast} \wtxp_{j}^{\ast} \theta_{j}
	\end{equation}
	on $\wtM$ and set
	\begin{equation}
		\calA(\wtxth_{j})
		\coloneqq \Vol_{\wtxth_{j}}(F)^{- n / (n + 1)} \int_{F} R_{\wtxth_{j}} \vol_{\wtxth_{j}}.
	\end{equation}
	On the one hand,
	\begin{equation}
	\label{eq:total-scalar-curvature-on-fundamental-domain}
		\calA(\wtxth_{j})
		= Y(M, T^{1, 0} M) \Vol_{\wtxp^{\ast} \theta_{0}}(\Psi_{j}(F))^{1 / (n + 1)}
		> 0
	\end{equation}
	since $R_{\theta_{0}} = Y(M, T^{1, 0} M)$.
	On the other hand,
	the scale invariance of $\calA$ and
	the assumption $R_{\theta_{j}} = Y(M_{j}, T^{1, 0} M_{j})$ imply that
	\begin{align}
		\calA(\wtxth_{j})
		&= \Vol_{\tau_{j}^{\ast} \wtxp_{j}^{\ast} \theta_{j}}(F)^{- n / (n + 1)}
			\int_{F} R_{\tau_{j}^{\ast} \wtxp_{j}^{\ast} \theta_{j}}
			\vol_{\tau_{j}^{\ast} \wtxp_{j}^{\ast} \theta_{j}} \\
		&= Y(M_{j}, T^{1, 0} M_{j}) \Vol_{\tau_{j}^{\ast} \wtxp_{j}^{\ast} \theta_{j}}(F)^{1 / (n + 1)} \\
		&\leq Y(M_{j}, T^{1, 0} M_{j}) (m_{1} \dotsm m_{j})^{- 1 / (n + 1)}.
	\end{align}
	It follows from $0 < Y(M_{j}, T^{1, 0} M_{j}) \leq Y(S^{2 n + 1}, T^{1, 0} S^{2 n + 1})$ that
	$\calA(\wtxth_{j}) \to 0$ as $j \to \infty$,
	and so \cref{eq:total-scalar-curvature-on-fundamental-domain} gives
	$\Vol_{\wtxp^{\ast} \theta_{0}}(\Psi_{j}(F)) \to 0$ as $j \to \infty$.
	Hence the sequence $(\Psi_{j})_{j = 1}^{\infty}$ cannot be pre-compact
	in $\Aut_{\CR}(\wtM, T^{1, 0} \wtM)$.
	This and \cref{thm:Schoen's-theorem} yield that
	$(\wtM, T^{1, 0} \wtM)$ is CR diffeomorphic to $(\bbH^{2 n + 1}, T^{1, 0} \bbH^{2 n + 1})$.
	\cref{lem:CR-Yamabe-constant-Heisenberg-nilmanifold} implies that $Y(M, T^{1, 0} M) = 0$,
	which is a contradiction.
\end{proof}

In order to prove our main result,
we show that if $(M, T^{1, 0} M)$ has positive CR Yamabe constant
and its fundamental group has infinite profinite completion,
then we can construct a finite cover $M_{j}$ of $M$
such that $\# \GModuli(M_{j}, T^{1, 0} M_{j}) \geq j$ for any $j \in \bbZ_{> 0}$.

\begin{proposition}
\label{prop:existence-of-non-homothetic}
	Let $(M, T^{1, 0} M)$ be a compact, connected, strictly pseudoconvex CR manifold of dimension $2 n + 1$
	such that \Cref{keyassumption} holds and $\pi_{1}(M)$ has infinite profinite completion.
	Then there exist a regular infinite tower $(\pi_{j} \colon M_{j} \to M_{j - 1})_{j = 1}^{\infty}$
	of finite connected coverings over $M$
	and a sequence $(\theta_{j})_{j = 0}^{\infty}$
	of unit volume CR Yamabe minimizers on $(M_{j}, T^{1, 0} M_{j})$
	such that for any $j \in \bbZ_{> 0}$,
	the contact forms $((\Pi_{j}^{l})^{\ast} \theta_{l})_{l = 0}^{j}$ on $M_{j}$
	are pairwise non-homothetic
	and have positive constant Tanaka--Webster scalar curvature.
\end{proposition}

\begin{proof}
	We construct $M_{j}$ and $\theta_{j}$ by induction.
	Assume that $\Pi_{j}^{0} \colon M_{j} \to M$ is a regular finite connected covering
	and $\theta_{j}$ is a unit volume CR Yamabe minimizer on $(M_{j}, T^{1, 0} M_{j})$.
	Note that $R_{\theta_{j}} = Y(M_{j}, T^{1, 0} M_{j}) > 0$
	since $\Pi_{j}^{0}$ is a finite covering.
	Take a normal subgroup $\Gamma \triangleleft \pi_{1}(M)$ of finite index
	such that $\Gamma \subset \pi_{1}(M_{j})$ and
	\begin{equation}
		[\pi_{1}(M_{j}) : \Gamma]
		> \rbra*{\frac{Y(S^{2 n + 1}, T^{1, 0} S^{2 n + 1})}{Y(M_{j}, T^{1, 0} M_{j})}}^{n + 1};
	\end{equation}
	the existence of such $\Gamma$ follows from the assumption that
	$\pi_{1}(M)$ has infinite profinite completion.
	Then $\pi_{j + 1} \colon M_{j + 1} \coloneqq \wtM / \Gamma \to M_{j}$ is a finite connected covering
	and $\Pi_{j + 1}^{0} \coloneqq \Pi_{j} \circ \pi_{j + 1} \colon M_{j + 1} \to M$ is regular.
	Take a unit volume CR Yamabe minimizer $\theta_{j + 1}$ on $(M_{j + 1}, T^{1, 0} M_{j + 1})$,
	whose existence is guaranteed by \Cref{keyassumption}.
	This contact form satisfies
	\begin{align}
		\frakF(\pi_{j + 1}^{\ast} \theta_{j})
		&= Y(M_{j}, T^{1, 0} M_{j}) \Vol_{\pi_{j + 1}^{\ast} \theta_{j}}(M_{j + 1})^{1 / (n + 1)} \\
		&= Y(M_{j}, T^{1, 0} M_{j}) [\pi_{1}(M_{j}) : \Gamma]^{1 / (n + 1)} \\
		&> Y(S^{2 n + 1}, T^{1, 0} S^{2 n + 1}) \\
		&\geq Y(M_{j + 1}, T^{1, 0} M_{j + 1})
		= \frakF(\theta_{j + 1}).
	\end{align}
	This construction yields
	a regular infinite tower $(\pi_{j} \colon M_{j} \to M_{j - 1})_{j = 1}^{\infty}$
	of finite connected coverings over $M$
	and a sequence $(\theta_{j})_{j = 0}^{\infty}$
	of unit volume CR Yamabe minimizers on $(M_{j}, T^{1, 0} M_{j})$
	such that $\frakF(\theta_{j + 1}) < \frakF(\pi_{j + 1}^{\ast} \theta_{j})$.
	It remains to show that
	$((\Pi_{j}^{l})^{\ast} \theta_{l})_{l = 0}^{j}$ are pairwise non-homothetic
	for each $j \in \bbZ_{> 0}$.
	If $0 \leq l \leq j - 1$,
	then
	\begin{align}
		\frakF((\Pi_{j}^{l})^{\ast} \theta_{l})
		= \frakF((\Pi_{j}^{l + 1})^{\ast} \pi_{l + 1}^{\ast} \theta_{l})
		&= \frakF(\pi_{l + 1}^{\ast} \theta_{l}) (\deg \Pi_{j}^{l + 1})^{1 / (n + 1)} \\
		&> \frakF(\theta_{l + 1}) (\deg \Pi_{j}^{l + 1})^{1 / (n + 1)} \\
		&= \frakF((\Pi_{j}^{l + 1})^{\ast} \theta_{l + 1}).
	\end{align}
	This yields
	\begin{equation}
		\frakF(\theta_{j})
		< \frakF((\Pi_{j}^{j - 1})^{\ast} \theta_{j - 1})
		< \dots
		< \frakF((\Pi_{j}^{0})^{\ast} \theta_{0}).
	\end{equation}
	It follows from the scale invariance of $\frakF$ that
	the contact forms $((\Pi_{j}^{l})^{\ast} \theta_{l})_{l = 0}^{j}$
	are pairwise non-homothetic on $M_{j}$.
\end{proof}

The lift $\wtxth_{j}$ of each contact form $\theta_{j}$ to the universal cover $\wtM$
is a complete contact form with constant Tanaka--Webster scalar curvature,
but these could be homothetic on $\wtM$.
However,
there exists a subsequence of $(\wtxth_{j})_{j = 0}^{\infty}$ that are pairwise non-homotheic on $\wtM$.

\begin{proof}[Proof of \cref{thm:infinite-geometric-moduli-general}]
	By \cref{prop:existence-of-non-homothetic},
	we can take a regular infinite tower $(\pi_{j} \colon M_{j} \to M_{j - 1})_{j = 1}^{\infty}$
	of finite connected coverings over $M$
	and a sequence $(\theta_{j})_{j = 0}^{\infty}$
	of unit volume CR Yamabe minimizers on $(M_{j}, T^{1, 0} M_{j})$.
	Then $\wtxth_{j} \coloneqq \wtxp_{j}^{\ast} \theta_{j}$ is a contact form on $\wtM$
	satisfying $R_{\wtxth_{j}} = Y(M_{j}, T^{1, 0} M_{j})$.
	Suppose to the contrary that
	\begin{equation}
		\Set{[\wtxth_{j}] | j = 0, 1, \dots}
		\subset \GModuli(\wtM, T^{1, 0} \wtM)
	\end{equation}
	is finite.
	Taking a subsequence,
	we may pick
	$\Phi_{j} \in \Aut_{\CR}(\wtM, T^{1, 0} \wtM)$
	and $c_{j} \in \bbR_{> 0}$ such that
	\begin{equation}
		\Phi_{j}^{\ast} \wtxp^{\ast} \theta_{0}
		= \Phi_{j}^{\ast} \wtxth_{0}
		= c_{j} \wtxth_{j}
		= c_{j} \wtxp_{j}^{\ast} \theta_{j}.
	\end{equation}
	This contradicts \cref{prop:existence-of-tower-implies-Heisenberg}.
\end{proof}

\section{Complements of spheres}
\label{section:complements-of-spheres}

\subsection{Complement of the $\bbR$-sphere}

Let $\bbB_{\bbR}^{k + 1} = \Set{x \in \bbR^{k + 1} | \abs{x}^{2} < 1}$
be the unit ball in $\bbR^{k + 1}$ with $0 \leq k \leq n - 1$.
Let $\Gamma_{0}$ be a torsion-free discrete subgroup of $PO(k + 1, 1) = SO_{+}(k + 1, 1)$
such that the quotient $\bbB_{\bbR}^{k + 1} / \Gamma_{0}$ is a compact manifold.
This assumption implies that
$\Gamma_{0}$ is a finitely generated infinite group.
We embed $SO_{+}(k + 1, 1)$ into $PU(n + 1, 1)$ by the composition of
\begin{equation}
	SO_{+}(k + 1, 1) \to U(n + 1, 1);
	\qquad
	A \mapsto
	\begin{pmatrix}
		A & 0 \\
		0 & I_{n - k}
	\end{pmatrix}
\end{equation}
and the canonical projection $U(n + 1, 1) \to PU(n + 1, 1)$.
Denote by $\Gamma$ the image of $\Gamma_{0}$ under this embedding.
The group $\Gamma$ is a torsion-free discrete subgroup of $PU(n + 1, 1)$
such that its limit set $\Lambda_{\Gamma}$ is the $\bbR$-sphere $S_{\bbR}^{k}$.
The closed convex hull $C_{\Gamma}$ of $\Lambda_{\Gamma}$ coincides with
the totally geodesic and totally real submanifold $\bbB_{\bbR}^{k + 1} \times \{0\}$ in $\bbB_{\bbC}^{n + 1}$.
Thus $C_{\Gamma} / \Gamma$ is diffeomorphic to $\bbB_{\bbR}^{k + 1} / \Gamma_{0}$,
which is compact.
Hence $\Gamma$ is convex cocompact and
\begin{equation}
	\delta_{\Gamma} = \dim_{H} \Lambda_{\Gamma} = k < n.
\end{equation}

\begin{proof}[Proof of \cref{thm:infinite-geometric-moduli-real-hyperbolic}]
	Since $S^{2 n + 1} \setminus S_{\bbR}^{k}$ is simply connected,
	$\pi_{1}(M_{\Gamma}) \cong \Gamma$
	and $\wtM_{\Gamma}$ is CR diffeomorphic to $S^{2 n + 1} \setminus S_{\bbR}^{k}$.
	In particular,
	the developing map of $M_{\Gamma}$ is injective.
	We need to check that \Cref{keyassumption} holds.
	It follows from $\delta_{\Gamma} < n$
	that $Y(M_{\Gamma}, T^{1, 0} M_{\Gamma}) > 0$.
	As we noted,
	the developing map of $M_{\Gamma}$ is injective.
	Moreover,
	$M_{\Gamma}$ is universally embeddable
	since any finite connected cover of $M_{\Gamma}$ is of the form $M_{\Gamma^{\prime}}$
	for a subgroup $\Gamma^{\prime} \subset \Gamma$ of finite index.
	Furthermore,
	$\pi_{1}(M_{\Gamma})$ has infinite profinite completion
	since it is infinite and residually finite according to \cref{prop:residually-finite-PU}.
	\cref{thm:infinite-geometric-moduli-general} implies that
	$\# \GModuli(S^{2 n + 1} \setminus S_{\bbR}^{k}) = \infty$.
\end{proof}

\subsection{Complement of the $\bbC$-sphere}

Assume that $n \geq 3$.
Let $\Gamma_{0}$ be a torsion-free discrete subgroup
of $U(k + 1, 1)$ with $0 \leq k < (n - 2) / 2$
such that the quotient $\bbB_{\bbC}^{k + 1} / \Gamma_{0}$ is a compact manifold.
This assumption implies that
$\Gamma_{0}$ is a finitely generated infinite group.
We embed $U(k + 1, 1)$ into $PU(n + 1, 1)$ by the composition of
\begin{equation}
	U(k + 1, 1) \to U(n + 1, 1);
	\qquad
	A \mapsto
	\begin{pmatrix}
		A & 0 \\
		0 & I_{n - k}
	\end{pmatrix}
\end{equation}
and the canonical projection $U(n + 1, 1) \to PU(n + 1, 1)$.
Denote by $\Gamma$ the image of $\Gamma_{0}$ under this embedding.
The group $\Gamma$ is a torsion-free discrete subgroup of $PU(n + 1, 1)$
such that its limit set $\Lambda_{\Gamma}$ is the $\bbC$-sphere $S_{\bbC}^{2 k + 1}$.
The closed convex hull $C_{\Gamma}$ of $\Lambda_{\Gamma}$ coincides with
the totally geodesic complex submanifold $\bbB_{\bbC}^{k + 1} \times \{0\}$ in $\bbB_{\bbC}^{n + 1}$.
Thus $C_{\Gamma} / \Gamma$ is diffeomorphic to $\bbB_{\bbC}^{k + 1} / \Gamma_{0}$,
which is compact.
Hence $\Gamma$ is convex cocompact and
\begin{equation}
	\delta_{\Gamma} = \dim_{H} \Lambda_{\Gamma} = 2 k + 2 < n.
\end{equation}

\begin{proof}[Proof of \cref{thm:infinite-geometric-moduli-complex-hyperbolic}]
	Since $S^{2 n + 1} \setminus S_{\bbC}^{2 k + 1}$ is simply connected,
	$\pi_{1}(M_{\Gamma}) \cong \Gamma$
	and $\wtM_{\Gamma}$ is CR diffeomorphic to $S^{2 n + 1} \setminus S_{\bbC}^{2 k + 1}$.
	It follows from $\delta_{\Gamma} < n$
	that $Y(M_{\Gamma}, T^{1, 0} M_{\Gamma})$ is positive,
	which guarantees \Cref{keyassumption}.
	Moreover,
	$\pi_{1}(M_{\Gamma})$ has infinite profinite completion
	since it is infinite and residually finite according to \cref{prop:residually-finite-PU}.
	\cref{thm:infinite-geometric-moduli-general} implies that
	$\# \GModuli(S^{2 n + 1} \setminus S_{\bbC}^{2 k + 1}) = \infty$.
\end{proof}

\section{Other applications}
\label{section:other-applications}

\subsection{Application 1: Circle bundle over \Kahler manifold}

Let $Y$ be a compact, connected, complex manifold of complex dimension $n$
and let $(L, h)$ be a Hermitian holomorphic line bundle over $Y$
such that
\begin{equation}
	\omega = - \sqrt{- 1} \Theta_{h} = d d^{c} \log h
\end{equation}
is a \Kahler form on $Y$.
Consider the Sasakian manifold $(S, T^{1, 0} S, \theta)$ as in \cref{subsection:Sasakian-manifolds}.

\begin{theorem}
\label{thm:infinite-geometric-moduli-circle-bundle}
	If the scalar curvature $R_{\omega}$ of $\omega$ is positive
	and $\pi_{1}(Y)$ has infinite profinite completion,
	then $\# \GModuli(\wtS, T^{1, 0} \wtS) = \infty$.
\end{theorem}

\begin{proof} 
	We first check that $(S, T^{1, 0} S)$ satisfies \Cref{keyassumption}.
	It follows from $R_{\theta} = p^{\ast} R_{\omega} > 0$ that $Y(S, T^{1, 0} S) > 0$.
	If $n = 1$,
	then $R_{\omega} > 0$ implies $Y = \cps^{1}$,
	which contradicts the assumption that $\pi_{1}(Y)$ has infinite profinite completion.
	If $n = 2$ and $(S, T^{1, 0} S)$ is spherical,
	then the classification of compact Bochner-flat \Kahler manifolds
	(\cite{Kamishima1994}*{Corollary B} or \cite{Bryant2001}*{Corollary 4.17})
	implies that the universal cover of $Y$ is diffeomorphic to $\cps^{2}$,
	which contradicts the assumption that $\pi_{1}(Y)$ has infinite profinite completion.
	
	The long exact sequence of homotopy groups for the fiber bundle $p \colon S \to Y$
	and $\pi_{0}(S^{1}) = 0$ yield that
	$p_{\ast} \colon \pi_{1}(S) \to \pi_{1}(Y)$ is surjective.
	It follows from \cref{lem:sujective-infinite-profinite-completion} that
	$\pi_{1}(S)$ has infinite profinite completion.
	Therefore \cref{thm:infinite-geometric-moduli-general} gives
	$\# \GModuli(\wtS, T^{1, 0} \wtS) = \infty$.
\end{proof}

\begin{remark}
	Serre~\cite{Serre1958} established that
	every finite group is the fundamental group of a smooth complex projective variety.
	This naturally led to the question of whether infinite \Kahler groups retain finite-like properties,
	such as residual finiteness.
	Toledo~\cite{Toledo1993} answered this in the negative by constructing
	the first examples of smooth projective varieties with non-residually finite fundamental groups.
	However,
	a weaker version of Serre's question remains a major open problem:
	it is unknown whether every infinite \Kahler group admits a proper normal subgroup of finite index
	(or equivalently, a non-trivial finite quotient);
	see \cite{Amoros-Burger-Corlette-Kotschick-Toledo1996} for example.
	In fact,
	the profinite completion of every infinite \Kahler group currently known is infinite.
\end{remark}

\subsection{Application 2: The boundary of a Reinhardt domain}
\label{subsection:boundary-of-Reinhardt-domain}

Assume that $n \geq 2$.
Let $M$ be the boundary of the bounded Reinhardt domain
\begin{equation}
	\Omega
	\coloneqq \Set{ w = (w^{1}, \dots , w^{n + 1}) \in \bbC^{n + 1} | \rho(w) < 0 },
\end{equation}
where
\begin{equation}
	\rho(w) \coloneqq \sum_{j = 1}^{n + 1} (\log \abs{w^{j}})^{2} - 1.
\end{equation}
This $M$ has the contact form $\theta \coloneqq d^{c} \rho |_{M}$.
Consider the holomorphic map
\begin{equation}
	\Psi \colon \bbC^{n + 1} \to \bbC^{n + 1}; \qquad
	(z^{1}, \dots , z^{n + 1}) \mapsto (\exp z^{1}, \dots , \exp z^{n + 1}).
\end{equation}
The pull-back $\Psi^{*} \rho$ coincides with
\begin{equation}
	\wtxr(z)
	\coloneqq \sum_{j = 1}^{n + 1} (\Re z^{j})^{2} - 1,
\end{equation}
and the pre-image of $\Omega$ by $\Psi$ is the tube domain
\begin{equation}
	\wtxco
	\coloneqq \Set{ z = (z^{1}, \dots , z^{n + 1}) \in \bbC^{n + 1} |
		\sum_{j = 1}^{n + 1} (\Re z^{j})^{2} < 1 }.
\end{equation}
The holomorphic map $\Psi$ induces a CR map $\psi \colon \wtM \coloneqq \bdry \wtxco \to M$.
Note that $\psi \colon \wtM \to M$ is the universal covering map.
The authors~\cite{Case-Takeuchi2023}*{Section 8} proved that
$R_{\theta}$ is a positive constant and the Chern tensor is nowhere vanishing.
In particular,
$Y(M, T^{1, 0} M) > 0$ and $(M, T^{1, 0} M)$ is not spherical,
which implies that $(M, T^{1, 0} M)$ satisfies \Cref{keyassumption}.
The fundamental group $\pi_{1}(M)$ of $M$ is isomorphic to $\bbZ^{n + 1}$,
which has infinite profinite completion.
Therefore \cref{thm:infinite-geometric-moduli-general} implies $\# \GModuli(\wtM, T^{1, 0} \wtM) = \infty$.

\bibliography{my-reference,my-reference-preprint}

\end{document}